\documentclass[12pt]{article}
 \usepackage[margin=1in]{geometry} 
\usepackage{amsmath,amsthm,amssymb,amsfonts}
 \usepackage{amssymb}

\usepackage[utf8]{inputenc}
\usepackage[pagewise]{lineno}
\newcounter{cnstcnt}

\usepackage{hyperref}
 \usepackage{csquotes}
\usepackage[utf8]{inputenc}
\usepackage[english]{babel}
\usepackage{xcolor}
\usepackage{biblatex}
\title{Blow-up prevention by sub-logistic sources in 2d Keller-Segel system }
\author{Minh Le }
\date{\today}

\begin{document}
\maketitle
\begin{abstract}
This paper investigates the global existence of solutions to Keller-Segel systems with sub-logistic sources using the test function method. Prior work by \cite{Tian1} demonstrated that sub-logistic sources $f(u)=ru -\mu \frac{u^2}{\ln^p(u+e)}$ with $p\in(0,1)$ can prevent blow-up solutions for the 2D minimal Keller-Segel chemotaxis model. Our study extends this result by showing that when $p=1$, sub-logistic sources can still prevent the occurrence of finite time blow-up solutions. Additionally, we provide a concise proof for a result previously proven in \cite{Cao} that the equi-integrability of $\left \{ \int_\Omega u^{\frac{n}{2}}(\cdot,t) \right \}_{t\in (0,T_{\rm max})}$ can avoid blow-up.
\end{abstract}
\numberwithin{equation}{section}
\newtheorem{theorem}{Theorem}[section]
\newtheorem{lemma}[theorem]{Lemma}
\newtheorem{remark}{Remark}[section]
\newtheorem{Prop}{Proposition}[section]
\newtheorem{Def}{Definition}[section]
\newtheorem{Corollary}{Corollary}[theorem]
\allowdisplaybreaks
\section{Introduction} \label{intro}
In this paper, we consider the following chemotaxis model with sub-logistic sources in a smooth bounded domain $\Omega \subset \mathbb{R}^n$, where $n \geq 2$:
\begin{equation} \label{5.e.p.sub-logistic}
    \begin{cases}
        u_t = \Delta u -\nabla \cdot (u \nabla v)+ f(u)\\
        0 = \Delta v +u-v,
    \end{cases}
\end{equation}
where $r,\mu$ are positive parameters, and $f$ is a smooth function generalizing the sub-logistic and signal production source respectively,  
\begin{align} \label{logistic}
    f(u) = ru-\mu \frac{u^2}{\ln^p(u+e)} ,\quad \text{with } r\in \mathbb{R}, \mu>0, \text{and } p>0.
\end{align}
The system \eqref{5.e.p.sub-logistic} is complemented with nonnegative initial conditions in $W^{1,\infty}(\Omega)$ not identically zero:
\begin{align} \label{initial-data}
    u(x,0)=u_0(x), \qquad v(x,0)=v_0(x), \qquad \text{with } x\in \mathbb{R},
 \end{align}
and homogeneous Neumann boundary condition are imposed as follows:
\begin{equation} \label{boundary-data}
    \frac{\partial u}{\partial \nu } = \frac{\partial v}{\partial \nu }  = 0, \qquad x\in \partial \Omega,\, t \in (0, T_{\rm max}),
\end{equation}
where $\nu$ denotes the outward normal vector. \\
The study of chemotaxis, which is the phenomenon where cells or bacteria move towards a chemical signal, has been a topic of intense research since the 1970s. Chemotaxis plays a significant role in various fields, including predicting the formation of aggregations, navigating optimal paths in a complex network, and even in physics, such as particle interaction. Moreover, it presents an intriguing mathematical property known as the critical mass phenomenon. This phenomenon means that if the mass is strictly less than a certain number, solutions exist globally, while if the mass is strictly larger than that number, solutions blow up in finite time. When $f \equiv 0$, it was shown in \cite{NSY} that the critical mass equals $4\pi$ when $\Omega = B(0,1)$ and $8\pi$ when the initial data are non-negative and radial. However, in higher dimensions, this property no longer holds. Recent research, as reported in \cite{Winkler-2010}, has shown that a finite blow-up solution can be constructed in a smooth bounded domain, regardless of how small the mass is. \\
The logistic sources, $f(u):= ru-\mu u^2$, was introduced and studied in \cite{Tello+Winkler} that if $\mu >\frac{n-2}{n}$ then solutions exist globally and are bounded at all time in a convex open bounded domain $\Omega \subset \mathbb{R}^n$ where $n \geq 2$. In order word, if $\mu$ is sufficiently large, then the quadratic term $-\mu u^2$ ensures no occurrence of blow-up solutions in two spacial dimensional domain. This leads to a natural question that whether the term "$-\mu u^2$" is optimal to prevent blow-up solutions. However, it has been discovered in \cite{Tian1} that the answer is negative. To be specific, the "weaker" term $-\frac{\mu u^2}{\ln^p(u+e)}$ for $0<p<1$ is sufficient to avoid blow-up solutions for both elliptic-parabolic and fully parabolic minimal Keller-Segel chemotaxis models in a two spacial dimensional domain.
Our main work improve the previous finding by showing that $p=1$ can prevent blow-up solutions of the system \eqref{5.e.p.sub-logistic}. \\
Our analysis relies on a test function method and Moser iteration technique. It is proved in \cite{Cao} that if the family of $\left \{ \int_\Omega u^{\frac{n}{2}}(\cdot,t) \right \}_{t \in (0, T_{\rm max})}$ is equi-integrable, then solutions of \eqref{5.e.p.sub-logistic} when $f \equiv 0$ exist globally and remain bounded at all time. In this paper, we give another shorter proof in Proposition \ref{5.Pop.paper3} for that result as well as indicate that the equi-integrability is not optimal to prevent blow-up thank to de la Vallée-Poussin Theorem. Thereafter, we try to a find a suitable functional and establish a differential inequality to obtain a priori estimate for solutions of \eqref{5.e.p.sub-logistic} thank to the presence of the sub-logistic quadratic degradation term "$-\mu \frac{u^2}{\ln(u+e)}$". Notice that the key milestone in this study is the choice of the following functional:
\begin{equation} \label{lnln-func}
    y(t) = \int_\Omega u(\cdot,t) \ln (\ln (u(\cdot,t)+e)) .
\end{equation}
One can also try to examine a functional 
\[
y_k(t) = \int_\Omega u(\cdot,t) \ln^k(u(\cdot,t)+e)
\]
to find an appropriate $k$, however, there is no suitable $k$ satisfying the conditions that $\mu$ can be arbitrary small. In order word, this method leads to the choice of $k$, but it does require the largeness assumption for $\mu$. So the functional \eqref{lnln-func} enables us to overcome that obstacle to prove our main theorem as follows:

\begin{theorem} \label{5.thm.paper3}
Let $\mu>0$, and $\Omega \subset \mathbb{R}^2 $ be a bounded domain with smooth boundary. The system \eqref{5.e.p.sub-logistic} under the assumptions \eqref{logistic}, \eqref{initial-data}, and \eqref{boundary-data} admits a global bounded solution in $\Omega\times( 0,\infty)$ .
\end{theorem}

\section{Preliminaries}
The local existence and uniqueness of non-negative classical solutions to the system \eqref{5.e.p.sub-logistic} can be established by adapting and adjusting the fixed point argument and standard parabolic regularity theory. For further details, we refer the reader to \cite{Winkler-Horstmann, Tello+Winkler, Lankeit-2017}. For convenience, we adopt Lemma 4.1 from \cite{MW2022}.
\begin{lemma} \label{local-existence}
    Let $\Omega \subset \mathbb{R}^n$, where $n \geq 2$ be a bounded domain with smooth boundary, and suppose $r\in \mathbb{R}$, $\mu>0$, the conditions \eqref{initial-data}, and \eqref{boundary-data} hold. Then there exist $T_{\rm max}\in (0,\infty]$ and functions 
    \begin{equation}
        \begin{cases}
            u \in C^0 \left ( \Bar{\Omega}\times (0,T_{\rm max}) \right ) \cap C^{2,1} \left ( \Bar{\Omega}\times (0,T_{\rm max}) \right )\text{ and} \\
            v \in \bigcap_{q>2} C^0 \left ( [0,T_{\rm max}); W^{1,q}(\Omega
            ) \right )\cap C^{2,1} \left ( \Bar{\Omega}\times (0,T_{\rm max}) \right )
        \end{cases}
    \end{equation}
    such that $u>0$ and $v>0$ in $\Bar{\Omega}\times (0,\infty)$, that $(u,v)$ solves \eqref{5.e.p.sub-logistic} classically in $\Omega \times (0,T_{\rm max})$, and that  
    \begin{equation} \label{local-existence-extend}
        \text{if }T_{\rm max}<\infty, \quad \text{then } \limsup_{t\to T_{\rm max}} \left \{ \left \| u \right \|_{L^\infty(\Omega)} +\left \| u \right \|_{W^{1,\infty}(\Omega)} \right \} = \infty.
    \end{equation}
\end{lemma}
We will use several interpolation inequalities extensively in the following sections. To start, we present an extended version of the Gagliardo-Nirenberg interpolation inequality, which was established in \cite{Li+Lankeit}.
\begin{lemma}[Gagliardo-Nirenberg interpolation inequality ] \label{GN}
Let $\Omega$ be a  bounded and smooth domain of $\mathbb{R}^n$ with $n \geq 1$. Let $r \geq 1$, $0<q\leq p < \infty$, $s>0$. Then there exists a constant $C_{GN}>0$ such that 
\begin{equation*}
    \left \| f \right \|^p_{L^p(\Omega)}\leq C_{GN}\left ( \left \| \nabla f \right \|_{L^r(\Omega)}^{pa}\left \| f \right \|^{p(1-a)}_{L^q(\Omega)} +\left \| f \right \|^p_{L^s(\Omega)}
 \right )
\end{equation*}
for all $f \in L^q(\Omega)$ with $\nabla f \in (L^r(\Omega))^n$, and $a= \frac{\frac{1}{q}-\frac{1}{p}}{\frac{1}{q}+\frac{1}{n}-\frac{1}{r}} \in [0,1]$.
\end{lemma}
In \cite{Cao}, an interpolation inequality of Ehrling-type is utilized to show that the equi-integrability of the family $\left \{ \int_\Omega u^{\frac{n}{2}}(\cdot,t) \right \}_{t \in (0,T_{\rm max})}$ implies the uniform boundedness of solutions. In this paper, we present an interpolation inequality that is similar to \cite[Lemma 2.1]{Cao}, and which will be employed to obtain an $L^q$ estimate with $q\geq 2$ for the solutions of the system \eqref{5.e.p.sub-logistic}. To prove this inequality, we adapt the argument used in the proof of inequality (22) in \cite{Biler+Hebisch}, with some modifications. We include a complete proof of this interpolation inequality below for the reader's convenience.
\begin{lemma} \label{ine-GN}
    Let $\Omega \subset \mathbb{R}^n$, with $n \geq 2$ be a bounded domain with smooth boundary and $q>\frac{n}{2}$. Then one can find $C>0$ such that for each $\epsilon >0$, there exists $c(\epsilon )>0$ such that
\begin{equation} \label{unif-GN}
\int_\Omega |w|^{q+1} \leq \epsilon \int_\Omega |\nabla w^{\frac{q}{2}}|^2 \left ( \int_\Omega G(|w|^{\frac{n}{2}}) \right )^{\frac{2}{n}} +C \left ( \int_\Omega |w| \right )^{q+1} + c(\epsilon) \int_\Omega |w|
\end{equation}
holds for all $w^{\frac{q}{2}} \in W^{1,2}(\Omega)$, and $\int_\Omega G(|w|^\frac{n}{2}) <\infty$ where $G $ is continuous, strictly increasing and  nonnegative  in $[0,\infty)$ such that $\lim_{s\to \infty} \frac{G(s)}{s} = \infty $. 
\end{lemma}

\begin{proof}
    We call 
    \begin{equation}
     \xi(s)  = \left\{\begin{matrix}
 0&  |s| \leq N \\
 2(|s|-N)& N< |s| \leq 2N \\
 |s|& |s|>2N.  \\
\end{matrix}\right.
    \end{equation}
    One can verify that
    \begin{align}
        \int_\Omega ||w|-\xi(w)|^{q+1} \leq (2N)^q \int_\Omega |w|
    \end{align}
and,
\begin{align}
    \int_\Omega \xi(w)^{\frac{n}{2}} \leq \frac{N^{\frac{n}{2}}}{G(N^{\frac{n}{2}})} \int_\Omega G(|w|^{\frac{n}{2}}).
\end{align}
Notice that $|\nabla \left (  \xi(w) \right )^{\frac{q}{2}}|^2 \leq c|w|^{q-2}|\nabla w|^2$, for some $c>0$, and combine with Lemma \ref{GN}, we obtain
\begin{align}
    \int_\Omega (\xi(w))^{q+1} &\leq c \int_\Omega |\nabla (\xi(w))^{\frac{q}{2}}|^2 \left ( \int_\Omega \xi(w)^{\frac{n}{2}} \right )^{\frac{2}{n}}  +C \left(  \int_\Omega \xi(w) \right )^{q+1} \notag \\
    &\leq c \left (\frac{N^{\frac{n}{2}}}{G(N^\frac{n}{2})} \right)^\frac{2}{n} \int_\Omega |\nabla w^{\frac{q}{2}}|^2 \left ( \int_\Omega G(|w|^\frac{n}{2})  \right )^{\frac{2}{n}} +C\left(  \int_\Omega |w| \right )^{q+1}.
\end{align}
This leads to
\begin{align}
    \int_\Omega |w|^{q+1} &\leq c\left (  \int_\Omega|\xi(w)|^{q+1} +\int_\Omega |\xi(w)- |w||^{q+1}  \right ) \notag \\ &\leq \left (\frac{N^{\frac{n}{2}}}{G(N^\frac{n}{2})} \right)^\frac{2}{n}  \int_\Omega |\nabla w^{\frac{q}{2}}|^2  \left ( \int_\Omega G(|w|^\frac{n}{2})  \right )^{\frac{2}{n}}+C\left(  \int_\Omega |w| \right )^{q+1} +(2N)^q \int_\Omega |w|.
\end{align}
We finally complete the proof by choosing N sufficiently large such that $c \left (\frac{N^{\frac{n}{2}}}{G(N^\frac{n}{2})} \right)^\frac{2}{n} \leq \epsilon $.
\end{proof}
Let us recall de la Vallée-Poussin Theorem
\begin{lemma} \label{VPlemma}
    The family $\left \{ X_\alpha \right \}_{\alpha \in A} \subset L^1(\mu)$ is uniformly integrable if and only if there exists a non-negative increasing convex function $G(t)$ such that
    \begin{align*}
        \lim_{t\to \infty} \frac{G(t)}{t}=\infty \qquad \text{and } \sup_\alpha \int_\Omega G(X_\alpha) <\infty.
    \end{align*}
\end{lemma}

\section{A priori estimates and proof of main theorem}
In this section, $(u,v)$ is a classical solutions as defined in Lemma \ref{local-existence} to the system \eqref{5.e.p.sub-logistic} with $p=1$. Our aim is to establish a priori estimate for the solutions. While the method in \cite{OTYM} and \cite{Tian1} relies on the $L^1$-estimate of $u$ and the absorption of $-\int_\Omega |\nabla u^{\frac{1}{2}}|^2$ to obtain a $L \ln L$ uniform bound, we take advantage of the term $- \mu\frac{u^2}{\ln(u+e)}$ to obtain a weaker $L \ln \ln L$ uniform bound.
\begin{lemma} \label{5.prioriest.lm.paper3}
    There exists $C=C(u_0,v_0, |\Omega|,\mu )>0$ such that
    \begin{equation}
        \sup_{t \in (0,T_{\rm max})} \int_\Omega u(\cdot,t) \ln(\ln(u(\cdot, t)+e)) \leq C.
    \end{equation}
\end{lemma}

\begin{proof}
    We define $y(t)=\int_\Omega u\ln (\ln (u+e))$ and differentiate $y$ to obtain
    \begin{align} \label{5proof.thm.1}
        y'(t)&= \int_\Omega \left [ \ln (\ln (u+e))+\frac{u}{(u+e)\ln(u+e)} \right ]u_t \notag \\
        &=\int_\Omega \left [ \ln (\ln (u+e))+\frac{u}{(u+e)\ln(u+e)} \right ] \left ( \Delta u-\nabla\cdot (u\nabla v)+ru-\mu \frac{u^2}{\ln(u+e)}\right ) \notag \\
        &=-\int_\Omega \nabla \left [ \ln (\ln (u+e))+\frac{u}{(u+e)\ln(u+e)} \right ] \cdot \nabla u \notag \\
        &+ \int_\Omega u \nabla \left ( \ln (\ln(u+e)) +\frac{u}{(u+e)\ln(u+e)} \right )\cdot \nabla v \notag \\
        &+\int_\Omega \left [ \ln (\ln (u+e))+\frac{u}{(u+e)\ln(u+e)} \right ] \left ( ru -\mu \frac{u^2}{\ln(u+e)} \right ) \notag \\
        &:= I+J+K
    \end{align}
    By integration by parts, we have
    \begin{align}\label{5proof.thm.2}
        I &=-\int_\Omega \nabla \left [ \ln (\ln (u+e))+\frac{u}{(u+e)\ln(u+e)} \right ] \cdot \nabla u \notag \\
        &= -\int_\Omega \left [ \frac{1}{(u+e)\ln(u+e)} +\frac{e\ln(u+e)-u}{(u+e)^2\ln^2(u+e)}  \right ]|\nabla u|^2 \notag \\
        &=-\int_\Omega \frac{u\ln(u+e)+2e\ln(u+e)-u}{(u+e)^2\ln^2(u+e)}|\nabla u|^2 \leq 0.
    \end{align}
    Similarly, we have
    \begin{align}\label{5proof.thm.3}
        J&= \int_\Omega u \nabla \left ( \ln (\ln(u+e)) +\frac{u}{(u+e)\ln(u+e)} \right )\cdot \nabla v \notag \\
        &= \int_\Omega \frac{u^2(\ln(u+e)-1)+2e\ln(u+e)}{(u+e)^2\ln^2(u+e)} \nabla u \cdot \nabla v \notag \\
        &=\int_\Omega \nabla \phi(u) \cdot \nabla v=\int_\Omega \phi(u)(u-v)\leq\int_\Omega u\phi(u),
    \end{align}
    where
    \begin{align}\label{5proof.thm.4}
        0 \leq \phi(u):= \int_0^u \frac{s^2(\ln(s+e)-1)+2e\ln(s+e)}{(s+e)^2\ln^2(s+e)}\, ds \leq \int_0^u \frac{1}{\ln(s+e)}\, ds.
    \end{align}
 Thus, we obtain
 \begin{align}\label{5proof.thm.5}
     J\leq \int_\Omega u \int_0^u \frac{1}{\ln(s+e)}\,ds.
 \end{align}
 By L'Hospital lemma, we have
 \begin{align}\label{5proof.thm.6}
     \lim_{u\to \infty} \frac{\int_0^u \frac{1}{\ln(s+e)}\, ds}{\frac{u\ln(\ln(u+e))}{\ln(u+e)}}=\lim_{u\to \infty} \frac{\ln(u+e)}{\ln(u+e)\ln(\ln(u+e)) +\frac{u}{u+e} -\frac{u}{u+e} \ln(\ln (u+e))}=0.
 \end{align}
 Therefore, for any $\epsilon>0$, there exist $N$ depending on $\epsilon$ such that for $u>N$, we have
 \begin{align}\label{5proof.thm.7}
     \int_0^u \frac{1}{\ln(s+e)}\,ds \leq \epsilon u \frac{\ln(\ln(u+e))}{\ln(u+e)}.
 \end{align}
This leads to
\begin{align}\label{5proof.thm.8}
    \int_\Omega u \int_0^u \frac{1}{\ln(s+e)}\,ds &=\int_{u\leq N} u \int_0^u \frac{1}{\ln(s+e)}\,ds +\int_{u> N} u \int_0^u \frac{1}{\ln(s+e)}\,ds \notag \\
    & \leq \epsilon \int_\Omega u^2\frac{\ln(\ln(u+e))}{\ln(u+e)}+c
\end{align}
 where $c=N^2|\Omega|$. From \eqref{5proof.thm.5} and \eqref{5proof.thm.8}, we imply
 \begin{align}\label{5proof.thm.9}
     J\leq \epsilon \int_\Omega u^2\frac{\ln(\ln(u+e))}{\ln(u+e)}+c.
 \end{align}
One can verify that for any $\epsilon >0$, there exist $C(\epsilon)>0$ such that
 \begin{align}\label{5proof.thm.10}
     K&= \int_\Omega \left [ \ln (\ln (u+e))+\frac{u}{(u+e)\ln(u+e)} \right ] \left ( ru -\mu \frac{u^2}{\ln(u+e)} \right ) \notag \\
     &\leq (\epsilon -\mu) \int_\Omega u^2\frac{\ln(\ln(u+e))}{\ln(u+e)}+c
 \end{align}
 and 
 \begin{align}\label{5proof.thm.11}
     y(t)\leq \epsilon \int_\Omega u^2\frac{\ln(\ln(u+e))}{\ln(u+e)}+c.
 \end{align}
 Collect \eqref{5proof.thm.1}, \eqref{5proof.thm.2}, \eqref{5proof.thm.5}, \eqref{5proof.thm.9},\eqref{5proof.thm.10}, and \eqref{5proof.thm.11}, we have
 \begin{align}
     y'(t)+y(t) \leq (3\epsilon -\mu)\int_\Omega u^2\frac{\ln(\ln(u+e))}{\ln(u+e)}+c.
 \end{align}
 We choose $\epsilon$ sufficiently small and apply Gronwall's inequality to imply $y(t)\leq C$ for all $t>0$.
\end{proof}
Thank to Lemma \ref{VPlemma}, the equi-integrability of $\left \{ \int_\Omega u^{\frac{n}{2}}(\cdot,t) <\infty\right \}_{t\in (0, T_{\rm max})}$ is equivalent to \\ $\sup_{t \in (0,T_{\rm max})} \int_\Omega G(u^{\frac{n}{2}}(\cdot,t)) <\infty$ for some non-negative increasing convex function such that $\lim_{s \to \infty}\frac{G(s)}{s} =\infty$. However, the convexity condition is not necessary, which means that the equi-integrable condition can be relaxed. Indeed, following proposition gives us the $L^q$ bounds, where $q>\frac{n}{2}$ for solutions without the convexity assumption. 
\begin{Prop} \label{5.Pop.paper3} Let $\Omega \subset \mathbb{R}^n$, where $n \geq 2$, be a bounded domain with smooth boundary, and $f \in C^2([0,\infty))$  such that $f(s)\leq c(s^2+1)$ for all $s\geq 0$, where $c>0$. Assume that $(u,v)$ is a classical solution as in Lemma \ref{local-existence} of \eqref{5.e.p.sub-logistic} on $\Omega \times (0,T_{\rm max})$ with maximal existence time $T_{\rm max} \in (0, \infty]$. If there exists a nonnegative increasing function $G$ such that
 \begin{align*}
        \lim_{t\to \infty} \frac{G(s)}{s}=\infty \qquad \text{and } \sup_{t \in (0, T_{\rm max})} \int_\Omega G(u^{\frac{n}{2}}(\cdot,t)) <\infty,
    \end{align*}
then for any $q>\frac{n}{2}$ we have
\begin{align*}
    \sup_{t \in (0,T_{\rm max})} \int_\Omega u^q(\cdot ,t) <\infty.
\end{align*}
\end{Prop}
\begin{proof}
    We define 
    \[
    \phi(t):=\frac{1}{q} \int_\Omega u^q
    \]
and differentiate $\phi$ to obtain 
 \begin{align} \label{5.Pop.proof.1}
     \phi'(t) &=\int_\Omega u^{q-1} [\Delta u -\nabla \cdot (u\nabla v) +f(u)] \notag \\
     &= -c_1  \int_\Omega |\nabla u^{\frac{q}{2}}|^2 +c_2\int_\Omega u^{\frac{q}{2}} \nabla u^{\frac{q}{2}} \cdot \nabla v +c\int_\Omega u^{q+1} + u^{q-1}\notag \\
     &= I+J+K,
 \end{align}
 where $c_1,c_2$ are positive depending only on $q$. We make use of integration by parts and the second equation of \eqref{5.e.p.sub-logistic} to obtain
\begin{align}\label{5.Pop.proof.2}
  J&:=  c_2\int_\Omega u^{\frac{q}{2}} \nabla u^{\frac{q}{2}} \cdot \nabla v \notag 
  = -c_3 \int_\Omega u^q \Delta v \notag \\
  &= -c_3 \int_\Omega u^q(v-u) \leq c_3 \int_\Omega u^{q+1},  
\end{align}
where $c_3$ is positive depending only on $q$. By Young inequality, one can find $c_4=c_4(q)>0$, and $c_5=c_5(q,|\Omega|)>0$ such that
\begin{align}\label{5.Pop.proof.3}
    I+K +\phi \leq c_4\int_\Omega u^{q+1} +c_5.
\end{align}
We make use of Lemma \ref{ine-GN} to obtain that there exist $C>0$ such that for any $\epsilon>0$, there exists $c_6 =c_6(\epsilon)>0$ such that
\begin{align*}
    c_5 \int_\Omega u^{q+1} \leq \epsilon \int_\Omega |\nabla u^{\frac{q}{2}}|^2 \left ( \int_\Omega G(u^{\frac{n}{2}}) \right )^{\frac{2}{n}} +C \left ( \int_\Omega u \right )^{q+1}+c_6\int_\Omega u 
\end{align*}
This, together with the uniform
 bounded condition of $ \int_\Omega G(u^{\frac{n}{2}}(\cdot ,t))$ imply that
 \begin{align}\label{5.Pop.proof.4}
      c_4 \int_\Omega u^{q+1} \leq c_7\epsilon\int_\Omega |\nabla u^{\frac{q}{2}}|^2 +c_8,
 \end{align}
 where $c_7$ is positive independent of $\epsilon$ and $c_8=c_8(\epsilon)>0$. From \eqref{5.Pop.proof.1} to \eqref{5.Pop.proof.4}, we obtain that 
 \begin{align}
     \phi'(t)+\phi(t) \leq (c_7\epsilon-c_1)\int_\Omega |\nabla u^{\frac{q}{2}}|^2 +c_9,
 \end{align}
 where $c_9=c_5+c_8$. The proof is now completed by choosing $\epsilon <\frac{c_1}{c_7}$ and applying Gronwall's inequality.
\end{proof}
We are now ready to prove the main result.
\begin{proof}[Proof of Theorem \ref{5.thm.paper3}]
    From Lemma \ref{5.prioriest.lm.paper3}, we obtain that there exists $C_1>0$ such that 
    \[
   \sup_{t \in (0,T_{\rm max})} \int_\Omega G(u(\cdot,t)) \leq C_1,
    \]
 where $G(s):= s\ln(\ln(s+e))$, satisfying all conditions of Proposition \ref{5.Pop.paper3}. Therefore, we can apply Proposition \ref{5.Pop.paper3} to deduce that for any $q>1$ there exists $C_2=C_2(q)>0$ such that
 \[
 \sup_{t\in (0,T_{\rm max})} \int_\Omega u^q (\cdot,t) \leq C_2.
 \]
 This, together with the second equation and elliptic regularity theory imply that 
 \[
  \sup_{t\in (0,T_{\rm max})} \int_\Omega |\nabla v (\cdot,t)|^{2q} \leq C_3,
 \]
 for some $C_3=C_3(q)>0$. By applying Moser iteration procedure as in \cite{Alikakos1}, \cite{Alikakos2}, and \cite{Winkler-2011}, we obtain that 
 \begin{align*}
     \sup_{t\in (0,T_{\rm max })} \left \| u \right \|_{L^\infty(\Omega)} +\left \| v \right \|_{W^{1,\infty}(\Omega)} <\infty.
 \end{align*}
 This, combined with \eqref{local-existence-extend}, implies that $T_{\rm max} =  \infty$ and uniform boundedness of $(u,v)$.
\end{proof}

\section*{Acknowledgement}
The author is indebted to Professor Michael Winkler for his kindly assistance in providing insightful comments, suggestions and valuable references. Additionally, the author extends appreciation to Professor Zhengfang Zhou for thoroughly reviewing the manuscript, engaging in fruitful discussions, and bringing to attention certain errors throughout the course of this project.
\printbibliography

@article{Biler+Hebisch,
author = { Biler P, Hebisch W, Nadzieja T.},
year = {1994},
month = {},
pages = {1189-1209},
title = {The Debye system: existence and large time behavior of solutions},
volume = {23},
journal = {Nonlinear Analysis, Theory, Methods, and Applications},
}

@article{Winkler-2011,
author = {Tao, Y. and Winkler, M.},
year = {2011},
month = {06},
pages = {},
title = {Boundedness in a quasilinear parabolic-parabolic Keller-Segel system
with subcritical sensitivity},
volume = {252},
journal = {Journal of Differential Equations},
}

@article{Tian1,
author = {Tian Xiang},
year = {2017},
month = {12},
pages = {},
title = {Sub-logistic source can prevent blow-up in the 2D minimal Keller-Segel chemotaxis system},
volume = {59},
journal = {Journal of Mathematical Physics},
}

@article{Alikakos1,
author = {Alikakos, N.D.},
year = {1979},
month = {},
pages = {827–868},
title = { Lp bounds of solutions of reaction-diffusion equations},
volume = {4},
journal = {Comm.
Partial Differential Equations},
}

@article{Alikakos2,
author = {Alikakos, N.D.},
year = {1979},
month = {},
pages = {201-225},
title = { An application of the invariance principle to reaction diffusion
equations},
volume = {33},
journal = {J. Differential Equations },
}

@article{Winkler-2010,
author = {Winkler, M.},
year = {2010},
month = {},
pages = {2889–2905},
title = { Aggregation vs. global diffusive behavior in the higher-dimensional
Keller–Segel model},
volume = {248(12)},
journal = {J. Differential Equations},
}

@article{Winkler-Horstmann,
author = {Horstmann, D. and Winkler, M.},
year = {2005},
month = {},
pages = {52-107},
title = {Boundedness vs. blow-up in a chemotaxis system},
volume = {215},
journal = {J. Differential Equations},
}

@article{OTYM,
author = {K. Osaki and T. Tsujikawa and A. Yagi and M. Mimura},
year = {2002},
month = {},
pages = {119-144},
title = {Exponential attractor for a chemotaxis-growth system of equations},
volume = {51},
journal = {Nonlinear Analysis},
}

@article{MW2022,
author = {M. Winkler},
year = {2022},
month = {},
pages = { },
title = {A result on parabolic gradient regularity in Orlicz spaces
and application to absorption-induced blow-up prevention
in a Keller-Segel type cross-diffusion system
},
volume = { 1},
journal = {Preprint},
}

@article{Tello+Winkler,
author = {J. Tello and M. Winkler},
year = {2007},
month = {},
pages = {849-877},
title = {A chemotaxis system with logistic source,
},
volume = {32},
journal = {Comm. Partial Differential
Equations},
}

@article{Lankeit-2017,
author = {Johannes Lankeit},
year = {2017},
month = {},
pages = {4052-4048},
title = {Locally bounded global solutions to a chemotaxis consumption model with singular sensitivity and nonlinear diffusion,
},
volume = {262},
journal = {Journal of Differential Equations},
}

@article{NSY,
author = {Toshitaka Nagai and Takasi Senba and Kiyoshi Yoshida},
year = {1997},
month = {},
pages = {411-433},
title = {Application of the Trudinger-Moser inequality to a Parabolic System of Chemotaxis,
},
volume = {40},
journal = {Funkcilaj Ekvacioj},
}

@article{Cao,
author = {Xinru Cao},
year = {2018},
month = {},
pages = {},
title = {An interpolation inequality and its application in Keller-Segel model
},
volume = {},
journal = {	arXiv:1707.09235 },
}

@article{Li+Lankeit,
author = {Y. Li and J. Lankeit },
year = {2016},
month = {},
pages = {1564-1595},
title = {Boundedness in a chemotaxis-haptotaxis model with nonlinear diffusion
},
volume = {29(6)},
journal = {Nonlinearity.},
}
\end{document}